\documentclass[12pt]{amsart}
\usepackage{amsmath,amssymb,enumerate,blkarray,kbordermatrix}
\newtheorem{theorem}{Theorem}[section]
\newtheorem{claim}{}[theorem]
\newtheorem{lemma}[theorem]{Lemma}

\newtheorem{corollary}[theorem]{Corollary}
\newtheorem{conjecture}[theorem]{Conjecture}
\theoremstyle{definition}
\newtheorem{definition}[theorem]{Definition}

\renewcommand{\kbldelim}{(}
\renewcommand{\kbrdelim}{)}

\newcommand{\bF}{\mathbb F}
\newcommand{\bR}{\mathbb R}

\newcommand{\cF}{\mathcal{F}}

\newcommand{\cL}{\mathcal{L}}
\newcommand{\fH}{\mathfrak{H}}
\newcommand{\cM}{\mathcal{M}}
\newcommand{\cP}{\mathcal{P}}

\newcommand{\cU}{\mathcal{U}}

\DeclareMathOperator{\si}{si}

\DeclareMathOperator{\cl}{cl}

\DeclareMathOperator{\PG}{PG}
\DeclareMathOperator{\GF}{GF}

\newcommand{\elem}{\epsilon}
\newcommand{\del}{\setminus}
\newcommand{\con}{/}
\author{Peter Nelson}
\title[Growth Rate Functions]{Growth Rate Functions of Dense Classes of Representable Matroids}
\begin{document}
\maketitle
\begin{abstract}
	For each proper minor-closed subclass $\cM$ of the $\GF(q^2)$-representable matroids containing all simple $\GF(q)$-representable matroids, we give, for all large $r$, a tight upper bound on the number of points in a rank-$r$ matroid in $\cM$, and construct a rank-$r$ matroid in $\cM$ for which equality holds. As a consequence, we give a tight upper bound on the number of points in a $\GF(q^2)$-representable, rank-$r$ matroid with no $\PG(k,q^2)$-minor. 
\end{abstract}
\section{Introduction}

If $\cM$ is a class of matroids, then the \textit{growth rate function} $h_{\cM}$ of $\cM$ is the function whose value $h_{\cM}(n)$ at a nonnegative integer $n$ is defined to be the maximum of $|M|$, where $M$ is a simple matroid in $\cM$ with $r(M) \le n$, or to be $\infty$ if no such maximum exists. 

For each nonnegative integer $k$ and prime power $q$, let $\cP_{q,k}$ denote the set of matroids of the form $M \con C$, where $M$ is a $\GF(q^2)$-representable matroid, $C$ is a rank-$k$ independent set in $M$, and $M \del C$ is a projective geometry over $\GF(q)$. Equivalently, $\cP_{q,k}$ is the set of $\GF(q^2)$-representable, $k$-element projections of projective geometries over $\GF(q)$. We prove the following: 

\begin{theorem}\label{mainresult}
	Let $q$ be a prime power. If $\cM$ is a proper minor-closed subclass of the $\GF(q^2)$-representable matroids containing all simple $\GF(q)$-representable matroids, then there is an integer $k \ge 0$ such that $\cP_{q,k} \subseteq \cM$, and $h_{\cM}(n) = h_{\cP_{q,k}}(n)$ for all large $n$.
\end{theorem}

We also characterise the densest matroids in $\cP_{q,k}$, which will allow us to give an explicit expression for $h_\cM(n)$: 

\begin{theorem}\label{mainresult1}
	Let $q$ be a prime power. If $\cM$ is a proper minor-closed subclass of the GF($q^2$)-representable matroids containing all simple GF$(q)$-representable matroids, then there exist nonnegative integers $k_{\cM}$ and $n_{\cM}$ so that
	\[h_{\cM}(n) = \frac{q^{n+k_{\cM}}-1}{q-1} - q\left(\frac{q^{2k_{\cM}}-1}{q^2-1}\right)\]
for all $n \ge n_{\cM}$.
\end{theorem}

The qualitative behaviour of growth rate functions is elegantly summarised by the `Growth Rate Theorem', a combination of results of Geelen, Kabell, Kung and Whittle, proved in [\ref{gkw}]. All of our results treat classes of matroids satisfying condition (\ref{expgrowth}) of this theorem.

\begin{theorem}[Growth Rate Theorem]\label{growthrate}
	If $\cM$ is a minor-closed class of matroids, then either
	\begin{enumerate}
		\item There exists $c \in \bR$ so that $h_\cM(n) \le cn$ for all $n \ge 0$, or 
		\item $\cM$ contains all graphic matroids, and there exists $c \in \bR$ so that $h_\cM(n) \le cn^2$ for all $n \ge 0$, or
		\item\label{expgrowth} There is a prime power $q$, and $c \in \bR$, so that $\cM$ contains all $\GF(q)$-representable matroids and $h_\cM(n) \le cq^n$ for all $n \ge 0$. 
		\item\label{unboundedgrowth} $\cM$ contains all simple rank-$2$ matroids, and $h_{\cM}(n) = \infty$ for all $n \ge 2$. 
	\end{enumerate}
\end{theorem}

 Theorems~\ref{mainresult} and~\ref{mainresult1} are together equivalent to the following substantial refinement of the $\GF(q^2)$-representable case of the growth rate theorem, determining $h_{\cM}$ exactly for all large $n$ in the case where $\cM$ satisfies (\ref{expgrowth}). 
\begin{theorem}\label{growthraterefinement}
	If $q$ is a prime power, and $\cM$ is a proper minor-closed class of the $\GF(q^2)$-representable matroids, then either:
	\begin{enumerate}
		\item There exists $c \in \bR$ so that $h_\cM(n) \le cn$ for all $n \ge 0$, or 
		\item $\cM$ contains all graphic matroids, and there exists $c \in \bR$ so that $h_{\cM}(n) \le cn^2$ for all $n \ge 0$, or 
		\item There exists an integer $k \ge 0$ so that $\cP_{q,k} \subseteq \cM$, and $h_{\cM}(n)=\frac{q^{n+k}-1}{q-1} - q\left(\frac{q^{2k}-1}{q^2-1}\right)$ for all large $n$, or 
	\end{enumerate}
\end{theorem}

Another consequence of the characterisation of the densest matroids in $\cP_{q,k}$ is a bound on the density of a $\GF(q^2)$-representable matroid with no $\PG(k,q^2)$-minor:

\begin{theorem}\label{mainresult2}
	Let $q$ be a prime power, and $k \ge 0$ be an integer. There is an integer $n_{k,q} \ge 0 $ so that if $M$ is a simple GF($q^2$)-representable matroid of rank at least $n_{k,q}$ with no $PG(k+1,q^2)$-minor, then 
	\[ |M| \le \frac{q^{r(M)+k}-1}{q-1}-q\left(\frac{q^{2k}-1}{q^2-1}\right).\]
	Moreover, this bound is the best possible. 
\end{theorem}

	The theory we establish imposes severe limitations on the extremal behaviour of exponentially dense classes of $\GF(q^2)$-representable matroids, and thus also gives some interesting corollaries regarding growth rate functions of naturally occuring classes of this sort.

\begin{theorem}\label{maincor1}
	Let $q$ be a prime power. There exists an integer $n_q \ge 0$ so that if $j \ge 3$ is an odd number, and $\cM$ is the class of matroids representable over both $\GF(q^2)$ and $\GF(q^j)$, then
	\[ h_{\cM}(n) = \frac{q^{n+1}-1}{q-1}-q \]
	for all $n \ge n_q$. 
\end{theorem}

  This second result gives an apparently uncountably large collection of minor-closed classes of matroids, all arising naturally from representability, whose growth rate functions together give a finite set. 

\begin{theorem}\label{maincor2}
	Let $q$ be a prime power. There is a finite set $\fH_q$ of integer-valued functions satisfying the following: let $\cF$ be a set of fields such that $\GF(q^2) \in \cF$, and all fields in $\cF$ have a proper $\GF(q)$-subfield, but not all fields in $\cF$ have a $\GF(q^2)$-subfield. If $\cM$ is the class of matroids representable over all fields in $\cF$, then $h_{\cM} \in \fH_q$. 
\end{theorem}

This suggests the following ambitious conjecture, which states that the collection itself is not uncountable, but finite. 

\begin{conjecture}
	Let $q$ be a prime power. There is a finite set $\mathfrak{M}_q$ of minor-closed classes of matroids satisfying the following: if $\cF$ is a set of fields such that $\GF(q^2) \in \cF$, all fields in $\cF$ have a proper $\GF(q)$-subfield, and not all fields in $\cF$ have a $\GF(q^2)$-subfield, then if $\cM$ is the class of matroids representable over all fields in $\cF$, then $\cM \in \mathfrak{M}_q$.
\end{conjecture}

All of our main results apply only in the $\GF(q^2)$-representable setting. However, adaptations of our techniques should apply more generally; we believe that exponentially dense growth rate functions should have similar behaviour for all minor-closed classes of matroids:

\begin{conjecture}\label{mainconj}
	If $\cM$ is a minor-closed class of matroids satisfying condition (\ref{expgrowth}) of Theorem~\ref{growthrate} for some prime power $q$, then there exists an integer $k \ge 0$ and an integer $d$ with $0 \le d \le \frac{q^{2k}-1}{q^2-1}$ such that	
	\[ h_{\cM}(n) = \frac{q^{n+k}-1}{q-1} - qd\] 
	for all large $n$. 
\end{conjecture}

This conjecture is motivated by the belief that the densest rank-$n$ matroids in a class of base-$q$-exponential density should be small projections of projective geometries over $\GF(q)$; the conjectured value for $h_{\cM}(n)$ is the number of points in a rank-$n$ matroid of this sort. 

The subtractive constant $-qd$ can take a range of values. This is a result of the fact that there are many different ways to take $k$-element projections of $\PG(n,q)$, giving rise to minor-closed classes with different growth rate functions. The largest and smallest possible values of $d$ are of particular interest, and we briefly discuss them here.

If $M_n$ is a matroid, and $e \in E(M_n)$, freely placed in the flat $E(M_n)$, satisfies $M_n \del e \cong \PG(n,q)$, then $M_n \con e$ is the $\textit{principal truncation}$ of $\PG(n,q)$. This is a special case of a projection, and the simple rank-$n$ matroid $M_n \con e$ satisfies $|M_n \con e| = \frac{q^{n+1}-1}{q-1}$. Closing the set $\{M_n: n \ge 0\}$ under minors gives a class $\cM$ of matroids with $h_{\cM}(n) = \frac{q^{n+1}-1}{q-1}$. This is an example of a class where $d$ takes the value zero. 

A class where $d = \tfrac{q^{2k}-1}{q^2-1}$ is the class $\cP_{q,k}$ of Theorem~\ref{mainresult}. In fact, the theorem essentially states that if $\cM$ contains only $\GF(q^2)$-representable matroids, then $d$ must take this value. 
This is a consequence of the fact that there is, up to isomorphism, a unique way to take a $\GF(q^2)$-representable, $k$-element projection of a projective geometry over $\GF(q)$ that is not also a $(k-1)$-element projection of such a geometry. For this reason, the $\GF(q^2)$-representable  case we are considering is qualitatively different from the general case, and some techniques we use will not be applicable to any proof of Conjecture~\ref{mainconj}.




\section{Preliminaries}

We assume familiarity with matroid theory, using as a base the notation of Oxley [\ref{oxley}].  Additionally, if $M$ is a matroid, we will write $|M|$ to denote $|E(M)|$, and $\elem(M)$ to denote $|\si(M)|$. Thus, $h_{\cM}(n) = \max\{\elem(M): M \in \cM, r(M) \le n\}$. A \textit{point} is a rank-$1$ flat, and a \textit{line} is a rank-$2$ flat. If $\ell \ge 1$ is an integer, then $\cU(\ell)$ denotes the class of matroids with no $U_{2,\ell+2}$-minor. 


The following beautiful theorem was proved by Kung in [\ref{kung}]:
\begin{theorem}\label{kungdensity}
	If $\ell \ge 2$ is an integer, and $M \in \cU(\ell)$ is a matroid, then 
	\[\elem(M) \le \frac{\ell^{r(M)}-1}{\ell-1}.\]
\end{theorem}
This next theorem was proved by Geelen and Kabell in [\ref{gk}], but not in this explicit form:  

\begin{theorem}\label{densitygk}
	There is a real-valued function $f_{\ref{densitygk}}(\beta,\ell,n)$ so that if $\ell \ge 2$ and $n \ge 1$ are integers, $\beta > 1$ is a real number, and $M \in \cU(\ell)$ is a matroid with $\elem(M) \ge f_{\ref{densitygk}}(\beta,\ell,n)\beta^{r(M)}$, then $M$ has a PG$(n-1,q)$-minor for some prime power $q > \beta$. 
\end{theorem}
\begin{proof}
	If $\beta \ge 2$, then let $q' = \lfloor \beta \rfloor$, and $f_{\ref{densitygk}}(\beta,\ell,n)$ to be the integer $\alpha$, depending on $q', n$ and $\ell$, given by Theorem 2.1 of [\ref{gk}]. If $\beta < 2$, then let $c = f_{\ref{densitygk}}(\beta,\ell,n)$ to be an integer large enough such that $c \beta^n \ge a n^m$ for all $n \ge 2$, where $a$ and $m$ are the integers given by Theorem 2.2 of [\ref{gk}]. The result follows from one of these two theorems.   
\end{proof}

A very similar lemma to the following result was proved in [\ref{gn}]. The proof we give is only different in that it deals with a larger range of values for $\mu$. 

\begin{lemma}\label{skewsubset}
	Let $\lambda, \mu$ be real numbers with $\lambda > 0$ and $\mu > 1$. Let $k \ge 0$ and $\ell \ge 2$ be integers, and let $A$ and $B$ be disjoint sets of elements in a matroid $M \in \cU(\ell)$ with $r_M(B) \le k$ and $\elem_M(A) > \lambda \mu^{r_M(A)}$. Then there is a set $A' \subseteq A$ that is skew to $B$ and satisfies $\elem_M(A') > \lambda \left(\frac{\mu-1}{\ell}\right)^{k}\mu^{r_M(A')}$.  
\end{lemma}
\begin{proof}
	We will prove the result by induction on $k$; our base case is when $r_M(B) = 1$. Let $e \in B$ be a nonloop. We may assume that $r(M) \ge 2$, that $A$ is minimal satisfying $\elem(M|A) > \lambda \mu^{r_M(A)}$, and that $E(M) = A \cup \{e\}$. Let $W$ be a flat of $M$ with $e \notin W$, so that $r_M(W) = r(M)-2$. Let $H_0, \dotsc, H_m$ be the hyperplanes of $M$ containing $W$, where $e \in H_0$. The sets $\{H_i - W: 1 \le i \le m\}$ form a partition of $E(M)-W$. Also, $\si(M \con W) \cong U_{2,m+1}$, so $m \le \ell$. 
	
	Minimality of $A$ gives $\elem_M(H_0 \cap A) \le \lambda \mu^{r(M)-1}$, so 
	\[\elem(M|(A-H_0)) > \lambda(\mu-1)\mu^{r(M)-1}.\]
	The union of the hyperplanes $H_1, \dotsc, H_m$ contains $E(M) - H_0$, so by a majority argument, there is some $1 \le i \le m$ such that 
	\[\elem_M(A \cap H_i) \ge \frac{1}{m}\elem(M|(A-H_0)) > \lambda\left(\frac{\mu-1}{\ell}\right)\mu^{r(M)-1}.\]
	Set $A' = A \cap H_i$. Now $A'$ is skew to $e$ and therefore to $B$, and $A'$ has the size we want, completing the base case. The result now follows from a standard inductive argument.
\end{proof}



\section{Unique Representations}

We make a diversion. Our goal in this section is to establish that if $A$ is a matrix with entries in a finite field $\bF$, then a submatrix of $A$ representing a projective geometry over a subfield of $\bF$ can be assumed to only have entries in this subfield. Theorem~\ref{pgunique} is likely equivalent to statements already well-known by projective geometers.

If $q$ is a prime power, we will write $\GF(q)$ for some canonical field with $q$ elements. If $\bF$ has $\GF(q)$ as a subfield, $M$ is an $\bF$-representable matroid, and $R$ is a restriction of $M$, then $R$ is a \textit{$\GF(q)$-represented restriction} of $M$ if there is an $\bF$-representation $A$ of $M$ such that $A[E(R)]$ has entries only in $\GF(q)$. We will consider the case when $\bF = \GF(q^2)$. 

Two matrices $A$ and $B$ with entries in a field $\bF$ are \textit{projectively equivalent} if there is a sequence of elementary row operations and column scalings of $A$ that gives $B$. We say that $B$ is obtained by applying a \textit{projective transformation} to $A$. If this is the case, then $M(A) = M(B)$. 

Theorem~\ref{pgunique} is closely related to the following:

\begin{theorem}[Fundamental Theorem of Projective Geometry]
	Let $q$ be a prime power, and $n \ge 1$ be an integer. The matroid $\PG(n,q)$ is uniquely $\GF(q)$-representable, up to projective equivalence and field automorphisms. 
\end{theorem}

We require two well-known results, one from matroid theory [\ref{bl}] and one from algebra [\ref{fields}]:

\begin{theorem}\label{binaryunique}
If $M$ is a binary matroid, and $\bF$ is a field, then $M$ has at most one $\bF$-representation, up to projective equivalence.
\end{theorem}
\begin{theorem}[Subfield Criterion]\label{subfieldunique}
Let $q$ be a prime power, and $k \ge 1$ be an integer. The field $\GF(q^k)$ has a unique subfield of order $q$. 
\end{theorem}

\begin{theorem}\label{pgunique}
	If $q$ is a prime power, $n \ge 3$ is an integer, and $\bF$ is a finite extension field of $\GF(q)$, then each representation of $\PG(n-1,q)$ over $\bF$ is projectively equivalent to a representation over $\GF(q)$. 
\end{theorem}
\begin{proof}
	
	Let $M \cong \PG(n-1,q)$, and $A$ be an $\bF$-representation of $M$; we may assume that $A$ has an $I_n$-submatrix.  We will show that there is a $\GF(q)$-subfield $\bF'$ of $\bF$, so that for any pair of distinct columns $u$ and $v$ of $A$, and $\omega \in \bF'$, the vector $u + \omega v$ is parallel to a column of $A$. This property is preserved by row operations and column scalings, so we will freely apply projective transformations to $A$.  
	
	Let $\{x_1,x_2,x_3\}$ be an independent set of size $3$ in $M$, and $e_1, e_2, e_3$ be the first three vectors in the standard basis of $\bF^n$. The matrix $B$ with column set $\{e_1, e_2, e_3, e_1-e_2,e_2-e_3,e_3-e_1\}$ is a $\bF$-representation of the cycle matroid of $K_4$, and $M$ has an $M(K_4)$-restriction with basis $\{x_1,x_2,x_3\}$, so we may assume by Theorem~\ref{binaryunique} that $A_{x_i} = e_i$ for each $i \in \{1,2,3\}$, and moreover that all columns of $B$ are columns of $A$.  
	
	Let $Z$ be the set of vectors in $\bF^n$ that are parallel to a column of $A$. Since $M \cong \PG(n-1,q)$ is modular, if $L_1$ and $L_2$ are rank-$2$ subspaces of $\bF^n$, each spanned by a pair of vectors in $Z$, and $w \in L_1 \cap L_2$, then $w \in Z$. For simplicity, we will refer to such subspaces as \textit{lines}, and write $\cl(v_1,v_2)$ for the subspace spanned by vectors $v_1,v_2 \in \bF^n$.
	
	For $(i,j) \in \{(1,2),(2,3),(3,1)\}$, let $L_{ij} = \cl(e_i,e_j)$, and $F_{ij} = \{\omega \in \bF: e_i + \omega e_j \in Z\}.$ 
	Since all lines in $\PG(n-1,q)$ have $q+1$ points, and the elements of $F_{ij}$ are in one-to-one correspondence with the points other than $u_j$ on the line $L_{ij}$, we have $|F_{ij}| = q$, and since the columns of $B$ are columns of $A$, the sets $F_{ij}$ contain $0$ and $-1$. 
	\begin{claim}\label{vsclaim1}
		$F_{12} = F_{23} = F_{31}$, and this set is closed under $\bF$-inverses.
	\end{claim}
	\begin{proof}[Proof of claim:]
		Let $\alpha \in F_{12}$. The lines $\cl(e_1 + \alpha e_2, e_3-e_1)$ and $L_{23}$ meet at a point parallel to $e_2 + \alpha^{-1}e_3$, so $\alpha^{-1} \in F_{23}$. The lines $\cl(e_2 + \alpha^{-1}e_3, e_1-e_2)$ and $L_{31}$ meet at a point parallel to $e_3 + \alpha e_1$, so $\alpha \in F_{31}$. Finally, the lines $L_{12}$ and $\cl(e_3 + \alpha e_1, e_2-e_3)$ meet at a point parallel to $e_1 + \alpha^{-1}e_2$, so $\alpha^{-1} \in F_{12}$. Now, $F_{12} = \{\alpha^{-1}: \alpha \in F_{12}\}$, and the inclusions established give $F_{12} \supseteq F_{23} \supseteq F_{31} \supseteq F_{12}$, giving the claim.
	\end{proof}
	Let $F = F_{12} = F_{23} = F_{31}$. This second claim, together with the first claim and the fact that $F$ contains $-1$ and $0$, implies that $F$ is a subfield of $\bF$. 
	\begin{claim}
		$F$ is closed under subtraction and multiplication in $\bF$. 
	\end{claim}
	\begin{proof}[Proof of claim:]
		Let $\alpha, \beta \in F$. To see closure under multiplication, observe that $\alpha \in F_{12}$, $\beta \in F_{23}$, so $e_1 + \alpha e_2$ and $e_2 + \beta e_3$ are both in $Z$. The lines $\cl(e_1,e_2 + \beta e_3)$ and $\cl(e_1 + \alpha e_2, e_3)$ meet at a point parallel to $e_1 + \alpha e_2 + \alpha\beta e_3$, so this vector is in $Z$. The line $\cl(e_1 + \alpha e_2 + \alpha \beta e_3,e_2)$ meets $L_{31}$ at $e_3 + (\alpha\beta)^{-1}e_1$, so $(\alpha\beta)^{-1} \in F_{31}$, giving $\alpha\beta \in F$ by the first claim. 
				
		We have $\alpha, \beta \in F_{12}$, so $e_1 + \alpha e_2$ and $e_1 + \beta e_2$ are both in $Z$. The lines $\cl(e_1 + \alpha e_2, e_2-e_3)$ and $\cl(e_1 + \beta e_2, e_3-e_1)$ meet at a point parallel to $e_1 + \beta e_2 + (\alpha-\beta)e_3$, and $\cl(e_2, e_1 + \beta e_2 + (\alpha - \beta)e_3)$ meets $L_{31}$ at a point parallel to $(\beta - \alpha)^{-1}e_1 + e_3$, so $(\alpha-\beta)^{-1} \in F_{31}$, giving $\alpha - \beta \in F$ by the first claim. 
	\end{proof}
	By these two claims, $F$ is a subfield of $\bF$. We know $|F| = q$, so Theorem~\ref{subfieldunique} implies that $F = \GF(q)$. We have therefore shown that for all $\omega \in \GF(q)$ and distinct elements $x_1, x_2 \in E(M)$, the vector $A_{x_1} + \omega A_{x_2}$ is parallel to a column of $A$. We may assume that all columns of $I_n$ are columns of $A$, so by repeated applications of this fact, it follows that all nonzero vectors in $\bF^n$ are parallel to a column of $A$, which implies the theorem.
	\end{proof}

This theorem has an important immediate corollary:

\begin{corollary}\label{framenice}
	If $q$ is a prime power, $M$ is a GF($q^2$)-representable matroid, and $R$ is a $\PG(r(M)-1,q)$-restriction of $M$, then $R$ is $\GF(q)$-represented in $M$.
\end{corollary}

\begin{lemma}\label{pgframe}
	Let $q$ be a prime power, $M$ be a GF$(q^2)$-representable matroid, and let $R$ be a $\PG(r(M)-1,q)$-restriction of $M$. If $e \in E(M)$ is a nonloop, and $e$ is not parallel to a point of $R$, then there is a unique line $L$ of $R$ so that $e \in \cl_M(L)$. 
\end{lemma}
\begin{proof}
	By Corollary~\ref{framenice}, there is a $\GF(q^2)$-representation $A$ of $M$ so that $A[E(R)]$ has entries only in $\GF(q)$. Let $e \in E(M \del R)$ be a nonloop, and $\omega \in \GF(q^2)-\GF(q)$. Since $\{1,\omega\}$ is a basis for $\GF(q^2)$ over $\GF(q)$, there are vectors $v,v' \in \GF(q)^n$ so that $A_e = v + \omega v'$. Since $R \cong \PG(r(M)-1,q)$, the vectors $u$ and $v$ are parallel to columns $A_f$ and $A_{f'}$ of $A[E(R)]$, so $e \in \cl_M(\{f,f'\})$, which is a line of $R$. By modularity of the lines of $R$, and the fact that $e$ is not a point of $R$, this line is unique.
	
\end{proof}

\section{The Extremal Matroids}\label{epgsection}

In this section, we define and investigate the a class of matroids which we will later show are the densest matroids in $\cP_{q,k}$. 

\begin{definition}\label{defepg}
	Let $q$ be a prime power, and $k$ and $n$ be integers with $0 \le k \le n$. Define a set $Z(n-1,q,k) \subseteq (\GF(q^2))^{n}$ by \[Z(n-1,q,k) = \left\{ (x \ y): x \in (\GF(q^2))^k, y \in (\GF(q))^{n-k}\right\}.\]
	 Let $A$ be a matrix whose set of columns is $Z(n-1,q,k)$. We denote by $\PG^{(k)}(n-1,q)$ any matroid isomorphic to $\si(M(A))$.
\end{definition}

The matroid $\PG^{(k)}(n-1,q)$ is a rank-$n$ projective geometry over $\GF(q)$, extended by some points from a projective geometry over $\GF(q^2)$. It is clear that $\PG^{(k)}(n-1,q)$ has rank $n$. For any integers $0 \le k \le n \le n'$, the matroid $\PG^{(k)}(n'-1,q)$ has a $\PG^{(k)}(n-1,q)$-restriction. 

The number of points in $\PG^{(k)}(n-1,q)$ is simple to determine. We will use this lemma freely:

\begin{lemma}\label{sizeepg}
	If $q$ is a prime power, and $k \ge 0$ and $n \ge k$ are integers, then
	\[|\PG^{(k)}(n-1,q)| = \frac{q^{n+k}-1}{q-1} - q\left(\frac{q^{2k}-1}{q^2-1}\right).\]
\end{lemma}
\begin{proof}
	Let $Z = Z(n-1,q,k)$ be the set, and $A$ be the matrix in definition~\ref{defepg}. Let 
	\[ Z_1 = \{(x \ y) \in Z: x \in (\GF(q^2))^k, y \in (\GF(q))^{n-k} - \{0\}\},\] and \[Z_2 = \{(x \ 0) \in Z: x \in (\GF(q^2))^k - \{0\}\}. \]
	So $Z = Z_1 \cup Z_2 \cup \{0\}$. Each $z \in Z_1$ is parallel to exactly $q-1$ elements of $Z$: those of the form $\alpha z: \alpha \in \GF(q) - \{0\}$. Each $z \in Z_2$ is parallel to exactly $q^2-1$ elements of $Z$: those of the form $\beta z: \beta \in \GF(q^2)-\{0\}$. We have 
	\begin{align*}
		|\PG^{(k)}(n-1,q)| &= \elem(M(A)) \\
					   &= \frac{|Z_1|}{q-1} + \frac{|Z_2|}{q^2-1}\\
					   &= \frac{(q^2)^k(q^{n-k}-1)}{q-1} + \frac{q^{2k}-1}{q^2-1}, 
	\end{align*}
	and the result follows by a calculation.
\end{proof}

\begin{corollary}\label{squarefieldpg}
	If $k \ge 0$ and $n > k$ are integers, then $\PG^{(k)}(n-1,q)$ has a $\PG(k,q^2)$-restriction.
\end{corollary}
\begin{proof}
	$\PG^{(k)}(n,q)$ has an $\PG^{(k)}(k,q)$-restriction; it thus suffices to show that $\PG^{(k)}(k,q) \cong \PG(k,q^2)$. 
	By Lemma~\ref{sizeepg}, 
		\begin{align*}
			|\PG^{(k)}(k,q)| &= \frac{q^{2(k+1)}-1}{q-1} - q\left(\frac{q^{2(k+1)}-1}{q^2-1}\right)\\
						 &= \frac{(q^2)^{k+1}-1}{q^2-1}\\
						 &= |\PG(k,q^2)|,
		\end{align*}
		and the result follows from the fact that $\PG^{(k)}(k,q)$ is a rank-$(k+1)$, $\GF(q^2)$-representable matroid. 
\end{proof}

This is the largest projective geometry over $\GF(q^2)$ that we can find as a minor of $\PG^{(k)}(n-1,q)$: 

\begin{lemma}\label{nobigpg}
	Let $q$ be a prime power, and $0 \le k \le n$ be integers. The matroid $\PG^{(k)}(n-1,q)$ has no $\PG(k+1,q^2)$-minor.
\end{lemma}
\begin{proof}
	We may assume that $n > k+1$. Let $M \cong \PG^{(k)}(n-1,q)$, and let $A$ be the matrix whose columns are the vectors in $Z(n-1,q,k)$, so $M = \si(M(A))$. The first $k$ standard basis vectors of $\GF(q^2)^n$ are columns of $A$, and contracting these columns gives a $\GF(q)$-representable matroid. Therefore, for any contraction-minor $M'$ of $M$, there is a set $C \subseteq E(M')$ of rank at most $k$ such that $M' \con C$ is $\GF(q)$-representable. Any matroid with a $\PG(k+1,q^2)$-restriction does not have this property, giving the lemma. 
\end{proof}

It is straightforward to see over which fields extended projective geometries are representable:

\begin{lemma}\label{anyfield}
		Let $q$ be a prime power, and $n \ge 3$ be an integer. If $\bF$ is a field with a proper $\GF(q)$-subfield, then $\PG^{(1)}(n-1,q)$ is $\bF$-representable, and if $\bF$ has no $\GF(q^2)$-subfield, then $\PG^{(2)}(n-1,q)$ is not $\bF$-representable.
	\end{lemma}
	\begin{proof}
		Let $\omega \in \bF - \GF(q)$. Let $A_{\bF,\omega}$ be a matrix, containing as columns all vectors in $\bF^n$ whose first entry lies in the set $\{\alpha\omega + \beta: \alpha, \beta \in \GF(q)\}$, and whose other entries lie in $\GF(q)$. It is straightforward to check that $M(A_{\bF,\omega})$ does not depend on $\bF$ or $\omega$. We may therefore assume that $\bF = \GF(q^2)$. The set of columns of $A_{\GF(q^2),\omega}$ is the set $Z(n-1,q,1)$ from Definition~\ref{defepg}, giving the first part of the lemma.
		
		Lemma~\ref{squarefieldpg} implies that the matroid $\PG^{(2)}(n-1,q)$ has a $\PG(2,q^2)$-restriction. This matroid admits no representation over a field without a $\GF(q^2)$-subfield. Therefore, if $\bF$ has no such subfield, $\PG^{(2)}(n-1,q)$ is not $\bF$-representable.
	\end{proof}



\section{Finding Extremal Matroids}

We give in this section a means to construct the extremal matroids of the previous section.

If $\cL$ is a set of lines in a matroid $M$, then $\cL$ is a \textit{matching} in $M$ if $r_M\left(\bigcup_{L \in \cL} L\right) = 2|L|$, or equivalently if the lines in $\cL$ are mutually skew in $M$. We define a new property in terms of a matching in a spanning $\PG(n,q)$-restriction. 
	
	\begin{definition}
		Let $q$ be a prime power, $M$ be a $\GF(q^2)$-representable matroid, and $R$ be a $\PG(r(M)-1,q)$-restriction of $M$. By Lemma~\ref{pgframe}, each nonloop of $e$ of $M$ is either parallel to a point of $R$, or there is a unique line $L_e$ of $R$ such that $e \in \cl_M(L_e)$. If $X \subseteq E(M)$ is an independent set of $M$ containing no point parallel to a point of $R$, and $\{L_e: e \in X\}$ is an $|X|$-matching in $R$, then we say that $X$ is \textit{$R$-unstable}. 
	\end{definition}
	
\begin{lemma}\label{contracttoepg}
	Let $q$ be a prime power, and let $k \ge 0$, $n \ge k$, and $n' \ge n+k$ be integers. If a rank-$n'$, GF$(q^2)$-representable matroid $M$ has a $\PG(n'-1,q)$-restriction $R$, and an $R$-unstable set of size $k$, then $M$ has a $\PG^{(k)}(n-1,q)$-minor. 
\end{lemma}
\begin{proof}
	We show that $\si((M \con X)|E(R)) \cong \PG^{(k)}(n'-1-k,q)$; the result will follow, as $n' - k \ge n$. 
	
	Let $X = \{e_1,\dotsc, e_k\}$ be an $R$-unstable set of size $k$, and for each $1 \le i \le k$, let $\{f_i, f_i'\}$ be a basis in $R$ of the unique line $L_i$ so that $e_i \in \cl_M(L_i)$. Since $\{f_1, \dotsc, f_k, f_1', \dotsc, f_k'\}$ is independent in $R$, and by Corollary~\ref{framenice}, there is a GF$(q^2)$-representation of $M|(X \cup E(R))$ of the following form: 
	\begin{align*}
 \kbordermatrix{ & X & &f_1 \cdots f_k &  f_1' \cdots f_k' & & E(R) - \{f_1, \dotsc, f_k, f'_1, \dotsc, f_k'\} &\\
                        & D & \vrule & I_k   & 0 &\vrule\\       
 & I_k & \vrule & 0 & I_k  & \vrule & Q \\  
   &0 &\vrule & 0& 0 & \vrule &    },
\end{align*}
where  $D$ is a $k \times k$ diagonal matrix whose diagonal entries are contained in $\GF(q^2)-\GF(q)$, and $M(A[E(R)]) \cong \PG(n'-1,q)$, with all entries of $Q$ in $\GF(q)$. Let $P$ be a matrix whose set of columns is $\GF(q)^{n'}$; every nonzero column of $P$ is parallel to some column of $A[E(R)]$. Let
\[A^+ = \kbordermatrix{&X &  & \\
& D & \vrule & P_1  \\
& I_k &\vrule & P_2 \\
&0 & \vrule & P_3 \\},\]
where $P^T = (P_1 \ P_2 \ P_3)^T$, and let $M^+ = M(A^+)$. By definition of $P$, we know that $\si(M^+) \cong M|(X \cup E(R))$, and we have \begin{align*}
M^+ \con X = M\left(
	\begin{array}{c}
		P_1 - DP_2 \\ P_3\\
	\end{array}\right).
\end{align*} For each diagonal entry $\omega$ of $D$, the field $\GF(q^2)$ is a vector space over $\GF(q)$ with basis $\{1,\omega\}$, so it follows from definition of $P$ and $D$ that the set of columns of $\binom{P_1 - DP_2}{P_3}$ is precisely the set $Z(n'-k-1,q,k)$ from Definition~\ref{defepg}. Therefore $\si(M^+ \con X) \cong \PG^{(k)}(n'-k-1,q)$. But $\si(M^+ \con X) \cong \si((M \con X)|E(R))$, so the result follows. 
\end{proof}

We now prove the important fact asserted at the beginning of the last section: that extended projective geometries are the densest matroids in $\cP_{q,k}$. 

\begin{lemma}\label{extendedproj}
	If $q$ is a prime power, and $n$ and $k$ are integers satisfying $0 \le k < n$, then every simple rank-$n$ matroid in $\cP_{q,k}$ is a restriction of $\PG^{(k)}(n-1,q)$, and $h_{\cP_{q,k}}(n) = |\PG^{(k)}(n-1,q)|$. 
\end{lemma}
\begin{proof}
	By Lemma~\ref{contracttoepg} applied when $n' = n+k$, the fact that $\PG^{(k)}(n-1,q) \in \cP_{q,k}$ is clear; therefore it suffices to show that every simple matroid $M \in \cP_{q,k}$ has a $\GF(q^2)$-representation in which all entries outside the first $k$ rows are in $\GF(q)$, as such a matroid is a restriction of $\PG^{(k)}(r(M)-1,q)$.
	
	Let $M \in \cP_{q,k}$; thus, let $M'$ be a $\GF(q^2)$-representable matroid, and $C = \{e_1, \dotsc, e_k\}$ be a rank-$k$ independent set in $M'$ with $M' \con C = M$ and $M' \del C \cong \PG(r(M'\del C)-1,q)$. By Lemma~\ref{framenice}, there is a representation $A$ of $M'$ in which all entries of $A[E(M)]$ are in $\GF(q)$. 
	
	Since $\GF(q^2)$ is a dimension-$2$ vector space over $\GF(q)$, we may apply a sequence of elementary row operations, scaling rows and columns only by elements of $\GF(q)$, to $A[C]$ so that all nonzero entries are in the first $2k$ rows. Applying these operations to $A$, and then contracting $C$, yields a representation of $M$ in which all entries outside the first $k$ rows are in $\GF(q)$, giving the result.
\end{proof}

Using the results established so far, we will prove Theorem~\ref{mainresult} by reducing it to the following theorem. We devote the remainder of our efforts to its proof. 

\begin{theorem}\label{getepg}
	There is an integer-valued function $f_{\ref{getepg}}(n,q,k)$ satisfying the following: if $q$ is a prime power, $n$ and $k$ are integers with $0 \le k < n$, and $M$ is a GF$(q^2)$-representable matroid satisfying $r(M) \ge f_{\ref{getepg}}(n,q,k)$ and 
	\[\elem(M) > |\PG^{(k)}(r(M)-1,q)|,\]
	then $M$ has an $\PG^{(k+1)}(n-1,q)$-minor.
\end{theorem}

\section{Matching in Projective Geometries}

	To construct the extremal matroids of the last two sections, we need to consider matchings in spanning projective geometries. The first theorem of this section follows easily from the linear matroid matching theorem of Lov\'asz ([\ref{lovasz}], Theorem 2), but is significantly weaker, and has a relatively short self-contained proof, which we include here. It gives a partly qualitative sufficient condition for the existence of a large matching.
	
	\begin{theorem}\label{pgmatching}
		There is an integer-valued function $f_{\ref{pgmatching}}(q,k)$ satisfying the following: if $q$ is a prime power, $n \ge 1$ and $k \ge 0$ are integers, and $M \cong \PG(n-1,q)$ is a matroid, then for any set $\cL$ of lines of $M$, either
		\begin{itemize}
			\item $\cL$ contains a $(k+1)$-matching of $M$, or
			\item There is a flat $F$ of $M$ with $r_M(F) \le k$, and a set $\cL_0 \subseteq \cL$ with $|\cL_0| \le f_{\ref{pgmatching}}(q,k)$, such that every line $L \in \cL$ either intersects $F$, or is in $\cL_0$. Moreover, if $r_M(F) = k$, then $\cL_0 = \varnothing$. 
		\end{itemize}
	\end{theorem}
	\begin{proof}
		Set 
		\[f_{\ref{pgmatching}}(q,k) = \frac{(q^{2k}-1)(q^{2k+3}-1)}{(q-1)^2} \]
		For every $e \in E(M)$, we write $\deg_\cL(e) = |\{L \in \cL: e \in L\}|$. Let $C \subseteq E(M)$ be a maximal independent set so that \[\deg_{\cL}(e) > \frac{q^{2k+3}-1}{q-1}\] for every $e \in C$. let $C' = C$ if $|C| \le k$, and $C'$ be a $(k+1)$-subset of $C$ otherwise. 
	\begin{claim}\label{matchingclaim1}
		$\cL$ contains a $|C'|$-matching. Moreover, if there is a line $L$ in $\cL$ skew to $C'$, then $\cL$ contains a $(|C'|+1)$-matching. 
	\end{claim}
	\begin{proof}[Proof of claim:]
		We prove the second part of the claim; the proof of the first part is similar but simpler. Let $|C'| = \{e_1, \dotsc, e_{|C'|}\}$. Let $j$ be maximal so that $0 \le j \le |C'|$, and  so that there is a $(j+1)$-matching $\cL_j = \{L, L_1, \dotsc, L_j\}$ so that $\cL_j \subseteq \cL$, and for each $1 \le i \le j$, we have $L_i \cap \cl_M(C') = \{e_i\}$. If $j = |C'|$, then $\cL_j$ satisfies the claim; we may therefore assume that $j < |C'|$. Since $\cL_j$ is a matching, and every line in $\cL_j - \{L\}$ meets $C'$ in a point, we have $r_M\left(C' \cup \bigcup_{L' \in L_j}(L')\right) = |C'|+2+j \le 2|C'|+1$.
		
		Since $\deg_{\cL}(e_{j+1}) > \frac{q^{2k+3}-1}{q-1} \ge \frac{q^{2|C'|+1}-1}{q-1}$, and $M$ is $\GF(q)$-representable, there is a set $X$ so that $\cl_M(\{x,e_{j+1}\}) \in \cL$ for all $x \in X$, and $r_M(X) > 2|C'|+1$. There is therefore some $x \in X$ not in $\cl_M(C' \cup \bigcup_{L' \in L_j}(L'))$. Now, $\cL_j \cup \{\cl_M(\{x,e_{j+1}\})\}$ is a matching of $M$, contradicting the maximality of $j$. 
	\end{proof}
	
	Suppose that the first outcome of the theorem does not hold; by \ref{matchingclaim1}, we may assume that $|C| \le k$. Let $\cL_0$ be the set of lines in $\cL$ that are skew to $C$.
	
	\begin{claim}
		$|\cL_0| \le f_{\ref{pgmatching}}(q,k)$.  
	\end{claim}
	\begin{proof}[Proof of claim:]
		By maximality of $C$, for each $e \notin \cl_M(C)$, we have $\deg_{\cL}(e) \le  \frac{q^{2k+3}-1}{q-1}$. Let $\cL_0'$ be a maximal matching contained in $\cL_0$, and let $F'$ be the flat spanned in $M$ by the lines in $\cL_0'$. We may assume that $|\cL_0'| \le k$, so $|F'| \le \frac{q^{2k}-1}{q-1}$. By maximality of $\cL_0'$ and modularity of $F'$, each $L \in \cL_0$ contains a point in $F'$, so the claim follows by this bound on $|F'|$, and our degree bound. 
	\end{proof}
	
	We now set $F = \cl_M(C)$. The flat $F$ is modular, so every line in $\cL - \cL_0$ meets $F$. If $r_M(F) = k$, and $L \in \cL_0$, then by~\ref{matchingclaim1}, $\cL$ contains a $(k+1)$-matching. So if $r_M(F) = k$, we must have $\cL_0 = \varnothing$. Now, $F$ and $\cL_0$ satisfy the second outcome of the lemma. 
	
	\end{proof}

	An easy application of this theorem allows us to find an unstable set:
	
	\begin{lemma}\label{findunstable}
		There is an integer-valued function $f_{\ref{findunstable}}(q,k)$ satisfying the following: if $q$ is a prime power, $k \ge 0$ is an integer, $M$ is a GF$(q^2)$-representable matroid, and $R$ is a $\GF(q)$-represented $\PG(r(M)-1,q)$-restriction of $M$, then either
		\begin{itemize}
			\item There is an $R$-unstable set of  size $k+1$ in $M$, or
			\item There is some $C \subseteq E(R)$ so that $r_M(C) \le k$, and $\epsilon(M \con C) \le \epsilon(R \con C) + f_{\ref{findunstable}}(q,k)$. 
		\end{itemize}
	\end{lemma}
	\begin{proof}
		Set $f_{\ref{findunstable}}(q,k) = (q^2+1)f_{\ref{pgmatching}}(q,k)$. We may assume that $M$ is simple; let $\cL$ be the set of lines $L$ of $R$ such that $|\cl_M(L)| > |\cl_R(L)|$. If $\cL$ contains a $(k+1)$-matching of $R$, then choosing a point from $\cl_M(L) - \cl_R(L)$ for each line $L$ in the matching gives an $R$-unstable set of size $k+1$. We may therefore assume that $\cL$ contains no such matching. Thus, let $F$ and $\cL_0$ be the sets defined in the second outcome of Theorem~\ref{pgmatching}. Let $C = F$, and $D = \cup_{L \in \cL_0} L$. We have $|D| \le (q^2+1)|\cL_0| \le f_{\ref{findunstable}}(q,k)$. By Lemma~\ref{pgframe}, each point of $M \del R$ lies in the closure of a line in $\cL$, so $\elem((M \con C) \del E(R)) \le \elem((M \con C)| D)$; the result now follows. 
	\end{proof}
	\begin{lemma}\label{finddistinctpoints}
			Let $q$ be a prime power, $d \ge 0$ be an integer, $M$ be a GF$(q^2)$-representable matroid, and $R$ be a $\PG(r(M)-1,q)$-restriction of $M$. If $\cL$ is a set of lines of $M$ so that $|L| > q+1$ for all $L \in \cL$, and $|\cL| > \binom{d+1}{2}$, then $\epsilon(M) > \epsilon(R) + d$.
	\end{lemma}		
	\begin{proof}
		We may assume that $M$ is simple; it therefore suffices to show that $|M \del R| > d$. By Corollary~\ref{framenice}, $R$ is $\GF(q)$-represented, and clearly, $L - E(R)$ is nonempty for every $L \in \cL$. For each $L \in \cL$, let $e_L \in L - E(R)$. Let $\cL_0 = \{L \in \cL: |L \cap E(R)| > 1\}$. Since $L \cap E(R)$ is a line of $R$ for each $L \in \cL_0$,  Lemma~\ref{pgframe} implies that the points $e_L: L \in \cL_0$ are distinct, so $\elem(M) \ge \elem(R) + |\cL_0|$. We may thus assume that $|\cL_0| \le d$, and therefore that $|\cL - \cL_0| > \binom{d+1}{2} - d = \binom{d}{2}$. 
		
		Now, each $L \in \cL - \cL_0$ contains at least two points of $M \del R$, and no two lines in $\cL - \cL_0$ contain two common points of $M \del R$, so it follows that $|\cL - \cL_0| \le \binom{|M \del R|}{2}$, and therefore that $|M \del R| > d$. 
	\end{proof}
\section{Weak Roundness}	

The results in this section concern the existence of dense, highly-connected
restrictions of large rank in dense matroids of very large rank.
These are similar to results in [\ref{gn}, Section 2]
where the notion of connectivity is \textit{roundness} (a matroid $M$ is round
if its ground set admits no partition into two sets of smaller rank than $M$). However, roundness has shortcomings
when the density is exponential with base $2$.
The rank-$r$ binary affine geometry has $2^{r-1}$
points and its only round restrictions have rank at most $1$.
This necessitates relaxing our connectivity notion.

	A matroid $M$ is \textit{weakly round} if $E(M)$ cannot be partitioned into sets $A$ and $B$ with $r(M|A) \le r(M)-1$ and $r(M|B) \le r(M)-2$. It is easy to check that this property is closed under both contraction and simplification. 
	
	Weak roundness is a vital property in our proof of Theorem~\ref{getepg}, and this section provides a means to reduce this theorem to the weakly round case; we prove that a dense matroid of very large rank has a similarly dense, weakly round restriction of large rank. 
	
	The following lemma is very similar to one that was proved in [\ref{gn}].
	\begin{lemma}\label{weakroundrestriction}
	Let $f(n)$ be an real-valued function satisfying $0 < f(1) \le f(2)$, and $f(n) \ge f(n-1) + f(n-2)$ for all $n > 2$. If $M$ is a matroid with $r(M) \ge 1$ and $\elem(M) > f(r(M))$, then $M$ has a weakly round restriction $N$ with $r(N) \ge 1$ and $\elem(N) > f(r(N))$. 
	\end{lemma}
	\begin{proof}
		We may assume that $M$ is not weakly round, so $r(M) > 2$, and there is a partition $(A,B)$ of $E(M)$ with $r_M(A) \le r(M)-1$ and $r_M(B) \le r(M)-2$. Inductively, we may assume that $\elem(M|A) \le f(r(M)-1)$ and $\elem(M|B) \le f(r(M)-2)$. So $\elem(M) \le f(r(M)-1) + f(r(M)-2) \le f(r(M))$, a contradiction.
	\end{proof}
	The next lemma contains the connectivity reduction that is key to our main proof. It is used in two distinct parts of the proof with respect to two different density functions, and is thus stated in an abstract way. 
	
	\begin{lemma}\label{weakroundnessreduction}
		There is a real-valued function $f_{\ref{weakroundnessreduction}}(\ell,\alpha,r)$ satisfying the following: if $\alpha > 0$ is a real number, $\ell \ge 2$ is an integer, and $g(n)$ is a real-valued function satisfying $g(1) \ge \alpha$, and $g(n) \ge 2g(n-1)$ for all $n > 1$,  and $M \in \cU(\ell)$ is a matroid satisfying $\elem(M) > g(r(M))$ and $r(M) \ge f_{\ref{weakroundnessreduction}}(\ell,\alpha,r)$, then $M$ has a weakly round restriction $N$ so that $\elem(N) > g(r(N))$, and $r(N) \ge r$.
	\end{lemma}
	\begin{proof}
		Set $f_{\ref{weakroundnessreduction}}(\ell,\alpha,r)$ to be large enough so that 
		\[ \alpha(\sqrt{5}-1)^{f_{\ref{weakroundnessreduction}}(\ell,\alpha,r)} \ge 2\frac{\ell^{r-1}-1}{\ell-1}. \]
		Let $\varphi = \frac{1}{2}(\sqrt{5}+1)$. Define a real-valued function $f$ by $f(n) = \varphi^{n-r(M)}\elem(M)$ for all integers $n$. We clearly have $f(n) = f(n-1) + f(n-2)$ for all $n$, and $0 < f(1) \le f(2)$. Also, it is clear that all $n \le r(M)$ satisfy $g(n) \ge 2^{n-1}g(1)$, and $g(n) \le 2^{n-r(M)}g(r(M)) < \varphi^{n-r(M)}\elem(M) = f(n)$. 
		
		By Lemma~\ref{weakroundrestriction}, there is a weakly round restriction $N$ of $M$ satisfying $\elem(N) > f(r(N))$ and $r(N) \ge 1$. We have $r(N) \le r(M)$, so $\elem(N) \ge f(r(N)) \ge g(r(N))$. Moreover, 
		\begin{align*}
			\elem(N) & \ge f(r(N)) \\
					 & >\varphi^{r(N)-r(M)}g(r(M)) \\
					 & \ge \varphi^{-r(M)}2^{r(M)-1}g(1) \\
					 & = \tfrac{1}{2}(2\varphi^{-1})^{r(M)}\alpha \\
					 &\ge \frac{\ell^{r-1}-1}{\ell-1}.
		\end{align*}
		By Lemma~\ref{kungdensity}, $r(N) \ge r$, implying the result.
	\end{proof}

\section{Exploiting Weak Roundness}

Our main result of this section is a technical lemma that uses the assumption of weak roundness to contract a set of bounded size onto a large projective geometry. This lemma contains most of the machinery in the proof of the main theorem of [\ref{gn}], and we have formulated it here in a more general sense than is required, to emphasise that $M^+$ need not be representable for the lemma to hold. The case where $M = M^+$ is an important specialisation.

\begin{lemma}\label{contractrestriction}
	There is an integer-valued function $f_{\ref{contractrestriction}}(n,q,t,\ell)$ so that the following holds: if $q$ is a prime power, $t \ge 0$, $n \ge 1$, and $\ell \ge 2$ are integers and matroids $M^+ \in \cU(\ell)$, $M$,  and a set $B \subseteq E(M^+)$, satisfy
	\begin{itemize}
	\item\label{cri} 
		$r_{M^+}(B) \le t$, and
	\item\label{crii} 
		$M$ is a weakly round, spanning restriction of $M^+$, and 
	\item\label{criii} 
		$M$ has a PG$(f_{\ref{contractrestriction}}(n,q,t,\ell)-1,q)$-minor $N$, 	 
	\end{itemize}
then there is a set $X \subseteq E(M)$, disjoint from $B$, so that $r(M \con X) \ge n$, and $M \con X$ has a $\PG(r(M \con X)-1,q)$-restriction, and $(M^+ \con X)|B = M^+|B$. 
\end{lemma}

\begin{proof}
	Let $n' = \max(n,t+1)$. Let $N = M \con C \del D$, where $C$ is independent in $M$. Let $\alpha = f_{\ref{densitygk}}(n',q-\tfrac{1}{2},\ell)$. Let $m$ be a positive integer large enough so that $m \ge 2t$, and so that \[\left(\frac{q}{q-\tfrac{1}{2}}\right)^m \ge \alpha q^2\left(\frac{\ell(q-\frac{1}{2})}{q-\frac{3}{2}}\right)^t,\] and \[\frac{q^m-1}{q-1} > q^{m-2} + \frac{\ell^t-1}{\ell-1}.\]Set $f_{\ref{contractrestriction}}(n,q,t,\ell) = m$. We may assume that $N \cong \PG(m-1,q)$. 
	
	\begin{claim}
	There is a set $C' \subseteq E(M)$, so that $M \con C'$ has a PG$(n'-1,q)$-restriction $N'$, and $(M^+ \con C')|B = M^+|B$.
	\end{claim}
	 \begin{proof}[Proof of claim:] Let $C_0 \subseteq C$ be maximal so that $(M^+ \con C_0)|B = M^+|B$, and let $M_0 = M \con C_0$, and $M^+_0 = M^+ \con C_0$. By maximality of $C_0$, we have $C - C_0 \subseteq \cl_{M^+_0}(B)$, and therefore $r_{M_0}(C-C_0) \le t$.  
	
	Let $A = E(N) - B$. We have \[r_{M'}(A) \le r_{M'}(E(N)) = r_{M' \con (C-C_0)}(E(N)) + r_{M'}(C-C_0) \le m+t.\] Since $M^+_0 \in \cU(\ell)$ and $r_{M^+_0}(B) \le t$, Theorem~\ref{kungdensity} gives $\elem_{M^+_0}(B) \le \frac{\ell^t-1}{\ell-1}$. Now
	\begin{align*}
		\elem_{M_0}(A) &\ge \elem_{M_0}(E(N)) - \elem_{M^+_0}(B) \\
				  &\ge \frac{q^m-1}{q-1} - \frac{\ell^t-1}{\ell-1}\\
				  & > q^{m-2}\\
				  & \ge \alpha \ell^t (q-\tfrac{3}{2})^{-t}(q - \tfrac{1}{2})^{m+t}\\
				  & \ge \alpha(\ell(q-\tfrac{3}{2})^{-1})^{t}(q - \tfrac{1}{2})^{r_{M_0}(A)}.
	\end{align*}  
	Applying Lemma~\ref{skewsubset} to $A$ and $B$ gives a set $A' \subseteq A$, skew to $B$ in $M^+_0$, satisfying $\elem_{M^+_0}(A') > \alpha(q-\tfrac{1}{2})^{r_{M^+_0}(A')}$. By Theorem~\ref{densitygk}, the matroid $M^+_0 |A' = M_0|A'$ has a PG$(n'-1,q')$-minor $N_1 = (M_0|A') \con C_1 \del D_1$ for some $q' > q-\tfrac{1}{2}$. 
	
	Since $A'$ is skew to $B$ in $M_0^+$, it is also skew to $C-C_0$, so $M_0|A' = (M_0 \con (C-C_0))|A' = N_1|A'$, and therefore $M|A'$ is GF$(q)$-representable, and so is $N_1$. So $q' = q$, and $N_1$ is a PG$(n'-1,q)$-restriction of $M_0 \con C_1$. Moreover, $C_1 \subseteq A'$, so $C_1$ is skew to $B$ in $M_0^+$, so $(M_0^+ \con C_1)|B = M_0^+|B = M^+|B$. Therefore, $C' = C_0 \cup C_1$ satisfies the claim. 
	\end{proof}
	
	Let $X$ be a maximal set satisfying the following:
	\begin{itemize}
		 \item $C' \subseteq X \subseteq E(M)-B$, and 
		 \item $(M^+ \con X)|B = M^+|B$, and 
		 \item $N'$ is a restriction of $M \con X$.
	\end{itemize}
	If $N'$ is spanning in $M \con X$, then $X$ satisfies the lemma. Otherwise, we have $r_{M^+}(B) \le t < n'= r(N') < r(M \con X)$. Weak roundness of $M \con X$ thus gives some $f \in E(M \con X)$ not in $\cl_{M \con X}(E(N'))$ or $\cl_{M^+ \con X}(B)$. This contradicts maximality of $X$.
\end{proof}

\section{The Spanning Case}\label{spanningsection}

	In this section, we show how to construct a $\PG^{(k+1)}$-minor directly from density in the case that we have a dense $\GF(q)$-represented restriction that is spanning and weakly round. 

	\begin{lemma}\label{spanningwin}
		There is an integer-valued function $f_{\ref{spanningwin}}(n,q,k)$ so that the following holds: if $q$ is a prime power, $n$ and $k $ are integers with $0 \le k < n$, and $M$ is a GF$(q^2)$-representable matroid such that 
		\begin{itemize}
			\item $M$ has a weakly round, spanning $\GF(q)$-represented restriction $R$, and 
			\item $R$ has a PG$(f_{\ref{spanningwin}}(n,q,k)-1,q)$-minor, and
			\item $\elem(M) > |\PG^{(k)}(r(M)-1,q)|$,
		\end{itemize}
		then $M$ has an $\PG^{(k+1)}(n-1,q)$-minor.
	\end{lemma}

	\begin{proof}
	Let $s$ be an integer so that \[|\PG^{(k)}(s'-1,q)| > |\PG^{(j)}(s'-1,q)| + f_{\ref{pgmatching}}(q,k)\] for all $j < k$ and $s' \ge s$. Set \[f_{\ref{spanningwin}}(n,q,k) = \max(s,f_{\ref{contractrestriction}}(n+k,q,2k+2,q^2)).\]
	
	We may assume that $M$ is simple. Let $A$ be a $\GF(q^2)$-representation of $M$ with $r(M)$ rows, so that $A[E(R)]$ has all entries in $\GF(q)$. Let $A'$ be a matrix formed by extending $R$ to a $\PG(r(M)-1,q)$-restriction $R'$ by appending columns with entries in $\GF(q)$ to $A$. Let $M' = M(A')$; by construction, $M'$ is simple, and $M$ is a spanning restriction of $M'$. 
	
	Let $\cL$ be the set of lines of $R'$, and let $\cL^+ = \{L \in \cL: \cl_{M'}(L) - E(R') \ne \varnothing\}$. Note that $|\cl_{M'}(L)| > q+1$ for all $L \in \cL^+$. Our goal is to use $\cL^+$ to find an unstable set in a minor. 
	
	\begin{claim} $\cL^+$ contains a $(k+1)$-matching of $R'$.
	\end{claim}
	\begin{proof}[Proof of claim:]
		Suppose not. Let $F \subseteq E(R')$ and $\cL_0 \subseteq \cL^+$ be the sets defined in Theorem~\ref{pgmatching}. Let $j = r_M(F)$; we know that $0 \le j \le k$, and if $j = k$, then $\cL_0 = \varnothing$. By Lemma~\ref{pgframe}, we have $E(M') = \left(\bigcup_{L \in \cL^+} L\right) \cup E(R')$. Let $\cL_F = \{L \in \cL: |L \cap F| =1\}$. So each point in $M' \del R'$ is either in $\cl_{M'}(F)$, in a line in $\cL_F$, or in a line in $\cL_0$.
		
		 Since $F$ is modular in $R'$, each point of $E(R') - F$ lies on $|F|$ distinct lines in $\cL_F$, and each line in $\cL_F$ contains exactly $q$ points in $E(R') - F$, so 
		\[|\cL_F| = \frac{|F|(|E(R')|-|F|)}{q} = \frac{(q^{j}-1)(q^{r(M)}-q^{j})}{q(q-1)^2}.\]
		
		Each line of $R'$ contains $q+1$ points of $R'$, and its closure in $M'$ contains at most $q^2-q$ points of $M' \del R'$. We can now estimate $\elem(M')$. 	
		\begin{align*}
			\elem(M') &= |R'| +  |M' \del R'| \\
					 &\le |R'| + \sum_{L \in \cL_F \cup \cL_0}|L-E(R')| +  |\cl_{M'}(F)-E(R')|\\
					 &\le \frac{q^{r(M)}-1}{q-1} + (q^2-q)(|\cL_F| + |\cL_0|)+ \left(\frac{q^{2j}-1}{q^2-1} - \frac{q^{j}-1}{q-1}\right)\\
					 &\le \frac{(q^2-q)(q^{j}-1)(q^{r(M)}-q^{j})}{q(q-1)^2} + \frac{q^{r(M)}-q^j}{q-1} + \frac{q^{2j}-1}{q^2-1} \\ &+ (q^2-q)|\cL_0|\\
					 &= \frac{q^{r(M)+j}-1}{q-1} - q\left(\frac{q^{2j}-1}{q^2-1}\right) + (q^2-q)|\cL_0|.\\
					 &= |\PG^{(j)}(r(M)-1,q)| + (q^2-q)|\cL_0|
		\end{align*}
		If $j < k$, then by the fact that $r(M') = r(M) \ge f_{\ref{spanningwin}}(n,q,k) \ge s$, we have $\elem(M') \le |\PG^{(k)}(r(M)-1,q)|$. If $j = k$, then $\cL_0 = \varnothing$, so $\elem(M') \le |\PG^{(k)}(r(M)-1,q)|$. In either case, \[\elem(M') \le |\PG^{(k)}(r(M)-1,q)| < \elem(M),\] contradicting the fact that $M$ is a restriction of $M'$. 
	\end{proof}
	Let $\{L_1, \dotsc, L_{k+1}\} \subseteq \cL^+$ be a $(k+1)$-matching, and let $B = \bigcup_{i=1}^{k+1}L_i$. We have $r_{M'}(B) = 2k+2$. The matroid $R$ is a weakly round, spanning restriction of $M'$, and $R$ has a $\PG(f_{\ref{contractrestriction}}(n+k,q,2k+2,q^2)-1,q)$-minor, so by Lemma~\ref{contractrestriction}, there is a set $X \subseteq E(R)$ so that $r(R \con X) \ge n+k$, and $R\con X$ has a $\PG(r(M \con X)-1,q)$-restriction $R_0$, and $(M' \con X)|B = M'|B$. 
	\begin{claim}
		$\si(M' \con X) \cong \si(M \con X)$.
	\end{claim}
	\begin{proof}[Proof of claim:]
		All entries of $A'[E(R')]$ are in GF$(q)$. In particular, the entries of $A'[X]$ are in GF$(q)$, so there is a $GF(q^2)$-representation $A_0$ of $M' \con X$ such that $A_0[E(R')-X]$ only has entries in GF$(q)$. 
		
		But $E(R_0) \subseteq E(R')-X$, and $R_0$ is a $\GF(q)$-represented $\PG(r(R \con X)-1,q)$-restriction of $R \con X$, so every column of $A_0$ with entries only in GF$(q)$ is parallel in $A_0$ to some element of $R_0$. All elements of $E(R')$ have this property, and $E(M') = E(M) \cup E(R')$, so the claim follows. 
	\end{proof}
	
	\begin{claim} There is an $R_0$-unstable set of size $k+1$ in $M' \con X$. 
	\end{claim}
	\begin{proof}[Proof of claim:]
		For each $1 \le i \le k+1$, let $L_i' = \cl_{M' \con X}(L_i)$. Since $(M' \con X)|B = M'|B$, the set $\{L_1', \dotsc, L_{k+1}'\}$ is a $(k+1)$-matching of $M' \con X$. Moreover, each $L_i$ is spanned by a pair of points of $R'$, and each such point is parallel in $M' \con X$ to a point of $R_0$, so for each $i$, the set $L_i' \cap E(R_0)$ is a line of $R_0$. Finally, $\elem(M' \con X | L_i') \ge \elem(M' | \cl_{M'}(L_i)) > q+1$ for each $i$, so each $L_i'$ contains a point $e_i$ not parallel to any points of $R_0$. The set $\{e_1, \dotsc, e_{k+1}\}$ is $R_0$-unstable in $M' \con X$. 
	\end{proof}
	
	By Lemma~\ref{contracttoepg}, the matroid $M' \con X$ has a $\PG^{(k+1)}(n-1,q,k+1)$-minor; by the second claim, so does $M \con X$.  
	\end{proof}
\section{Constellations}\label{constellationsection}

	If the hypotheses in the previous section fail, then we use a different method to find a $\PG^{(k)}(n,q)$-minor. 
	\begin{definition}\label{defconstellation}
		Let $s,\ell,j$ be positive integers. A matroid $K$ is an $(s,\ell,j)$-\textit{constellation} if 
		\begin{itemize}
			\item $r(K) \le s(j+1)$, and 
			\item $K$ has an independent set $S$ of size $s$ such that, for all $e \in S$, there exists an independent set $X_e$ of size $j$, such that, for all $f \in X_e$, the line $\cl_K(\{e,f\})$ contains at least $\ell+2$ points.
		\end{itemize}
	\end{definition}
	
	A constellation is an independent set of points, each of which is the centre of a `star' of an independent collection of $(\ell+2)$-point lines. If $K$ is any matroid satisfying the second part of the definition, then $K|\left(S \cup \bigcup_{e \in S}X_e\right)$ is an $(s,\ell,j)$-constellation. Moreover, for any $s' \le s$, an $(s,\ell,j)$-constellation has an $(s',\ell,j)$-constellation restriction, found by considering an $s'$-subset of $S$. 
	
	
		
	
	\begin{lemma}\label{constellationwin}
		There is an integer-valued function $f_{\ref{constellationwin}}(n,q,k)$ so that the following holds: if $q$ is a prime power, $n$ and $k $ are integers with $0 \le k < n$, and $M$ is a weakly round, GF($q^2$)-representable matroid with a $(f_{\ref{constellationwin}}(n,q,k),q,k+1)$-constellation restriction $K$, and a PG$(f_{\ref{constellationwin}}(n,q,k)-1,q)$-minor, then $M$ has a $\PG^{(k+1)}(n-1,q)$-minor.
	\end{lemma}
	\begin{proof}
		Let $d = f_{\ref{findunstable}}(q,k)$, and let $s = d(d+1)+k+1$. Set \[f_{\ref{constellationwin}}(n,q,k) = \max(s,f_{\ref{contractrestriction}}(n+k,q,s(k+2),q^2)).\] 
		
		
		Since $M$ is GF$(q^2)$-representable, we know that $M \in \cU(q^2)$. By Lemma~\ref{contractrestriction}, applied with $M^+ = M$, and $B = E(K)$, there is some set $X \subseteq E(M)$ so that $r(M \con X) \ge n+k$, and $M \con X$ has a $\PG(r(M \con X)-1,q)$-restriction $R$, and $(M \con X)|E(K) = M | E(K) = K$. Let $M' = M \con X$.
		
		\begin{claim} $M'$ has an $R$-unstable set of size $k+1$.
		\end{claim}
		\begin{proof}[Proof of Claim:]
			By Lemma~\ref{findunstable}, we may assume that there is a set $C \subseteq E(R)$ so that $r_{M'}(C) \le k$, and $\elem(M' \con C) \le \elem(R \con C) + d$. The set $S$ in the constellation $K$ is a rank-$(d(d+1)+k+1)$ set in $M'$; let $S' \subseteq S$ be an independent set of size $d(d+1)+1$ in $M' \con C$. Let $e \in S'$. Since $r_{M'}(X_e) > k$, there is some $f \in X_e$ so that $\{e,f\}$ is independent in $M' \con C$; let $L_e = \cl_{M' \con C}(\{e,f\})$. The line $L_e$ contains at least $q+2$ points in $K$, and therefore in $M' \con C$. 
			
		$S'$ is independent in $M' \con C$, so no line $L_e$ can contain more than two points of $S'$, giving $|\{L_e: e \in S'\}| \ge \frac{1}{2}|S'| > \binom{d+1}{2}$. The matroid $R \con C$ is a spanning restriction of $M' \con C$, and $\si(R \con C) \cong \PG(n-1-r_{M'}(C),q)$, so Lemma~\ref{finddistinctpoints} now implies that $\elem(M' \con C) > \elem(R \con C) + d$, a contradiction. 
		\end{proof}
		The lemma now follows from Lemma~\ref{contracttoepg}.
		
	\end{proof}

	\section{The Reductions}
	
	We will prove Theorem~\ref{getepg} by showing that it can be reduced to either Lemma~\ref{spanningwin} or Lemma~\ref{constellationwin}. The following technical lemma contains this reduction. 
	
	\begin{lemma}\label{mainreduction}
		There is an integer-valued function $f_{\ref{mainreduction}}(m,q,k)$ satisfying the following: if $q$ is a prime power, $m \ge 1$ and $k \ge 0$ are integers, and $M$ is a weakly round, GF$(q^2)$-representable matroid such that 
		\begin{itemize}
			\item $M$ has a PG$(f_{\ref{mainreduction}}(m,q,k)-1,q)$-minor, and 
			\item $\elem(M) > |\PG^{(k)}(r(M)-1,q)|$, 
		\end{itemize}
		then one of the following holds:
		\begin{enumerate}[(i)]
			\item\label{spanningoutcome}
				$M$ has a minor $M'$ such that 
				\begin{itemize}
					\item $M'$ has a weakly round, spanning $\GF(q)$-represented restriction $R$, and
					\item $R$ has a PG$(m-1,q)$-minor, and 
					\item $\elem(M') > |\PG^{(k)}(r(M')-1,q)|$,
				 
				\end{itemize}
				
				or
			\item\label{constellationoutcome}
				$M$ has a weakly round minor $M'$ with an $(m,q,k+1)$-constellation restriction, and a PG$(m-1,q)$-minor. 
		\end{enumerate}
	\end{lemma}
	\begin{proof}
		Let $r$ be an integer large enough so that 
		\[q^{r'-3m} \ge f_{\ref{densitygk}}(q-\tfrac{1}{2},q,m)(q - \tfrac{1}{2})^{r'}\]
		for all $r' \ge r$. Let $n = f_{\ref{weakroundnessreduction}}(q^2,q^{1-3m},r)+ 2m$. Set $f_{\ref{mainreduction}}(m,q,k) = n$. 
		
		We may assume that $M$ is simple, and minor-minimal satisfying the hypotheses. Let $N = M \con C \del D \cong \PG(n-1,q)$, where $C$ is independent, and $D$ is coindependent. 
		\begin{claim}\label{getconstclaim}
		$M$ has a $(|C|,q,k+1)$-constellation restriction. 
		\end{claim}
		\begin{proof}[Proof of claim:]\
			Let $e \in C$. The matroid $M \con e$ is weakly round and GF$(q^2)$-representable, and has an $N$-minor, so \[\elem(M \con e) \le |\PG^{(k)}(r(M\con e)-1,q)|\] by minor-minimality of $M$.  Let $\cL^+$ be the set of lines of $M$ containing $e$ and at least $q+1$ other points, and $\cL^{-}$ be the set of all other lines of $M$ containing $e$. Each line in $\cL^-$ contains at most $q$ points other than $e$, and each line in $\cL^+$ contains at most $q^2$ points other than $e$. We have $\elem(M \con e) = |\cL^+| + |\cL^-|$, and $\elem(M) \le q^2|\cL^+| + q|\cL^-| + 1 = q\elem(M \con e) + (q^2-q)|\cL^+| + 1$. Now
			\begin{align*}
				|\PG^{(k)}(r(M)-1,q)| &< \elem(M)	\\
								 &\le q\elem(M \con e) + (q^2-q)|\cL^+| + 1 \\
								 &\le q|\PG^{(k)}(r(M \con e)-1,q)| + (q^2-q)|\cL^+| + 1.
			\end{align*}
		This implies that \[|\cL^+| > \frac{1}{q^2-q}\left(|\PG^{(k)}(r(M)-1,q)|-q|\PG^{(k)}(r(M)-2,q)| - 1\right),\]
		and a computation gives $|\cL^+| > \frac{q^{2k}-1}{q^2-1}$. Let $X_e'$ be a set formed by choosing a point other than $e$ from each line in $\cL^+$. Since $M$ is GF($q^2$)-representable, it now follows that $r_M(X_e') > k$; let $X_e \subseteq X_e'$ be an independent set of size $k+1$. The set $C$, along with $X_e: e \in C$, gives the required constellation.
		\end{proof}
	Since $n \ge m$, the matroid $M$ also has a $\PG(m-1,q)$-minor, so if $|C| \ge m$, we have outcome~(\ref{constellationoutcome}) for $M$ by \ref{getconstclaim}. We may therefore assume that $|C| < m$. 
		
	\begin{claim}
		There is a weakly round, $\GF(q)$-represented restriction $R$ of $M$ so that $R$ has a $\PG(m-1,q)$-minor. 
	\end{claim}	
	\begin{proof}[Proof of claim:] 
		Since $E(N)$ is a spanning restriction of $M \con C$, there is a matrix $A'$ representing $M$ over GF$(q^2)$ of the following form:
		\[
		A' = 
		\begin{blockarray}{cccc}
			&C & E(N) & D\\
			\begin{block}{c(ccc)}
				C & I_C & Q_1 & Q_2 \\
				\left[n\right] & 0 & B & Q_3 \\
			\end{block}
		\end{blockarray}
		\]
		where $M(B) \cong \PG(n-1,q)$. By applying Theorem~\ref{pgunique} to the submatrix $A'[[n],E(M)]$, we may assume that all entries of $B$ are in GF$(q)$. Since $|C| < m$, there are at most $q^{2(m-1)}$ distinct column vectors in $Q_1$, so there is some $Y \subseteq E(N)$ so that $|Y| \ge q^{-2(m-1)}|E(N)|$, and all columns of the matrix $Q_1[Y]$ are the same. Now, 
		\[
			A'[Y] = \left(\begin{array}{c}
			Q_1[Y] \\ B[Y] \\
			\end{array}\right),
		\]
		where $Q_1[Y]$ is a matrix of rank at most $1$, so by scaling the first $|C|$ rows of $A'[Y]$, we can obtain a matrix of the following form:
		\[
			\left(\begin{array}{c} P \\ B[Y]\\ \end{array}\right),
		\]
		where all entries of $P$ are $0$ or $1$. Applying these same row scalings to $A'$ gives a matrix $A$ representing $M$ over GF$(q^2)$, in which all entries of $A[Y]$ are in $P$ or $B[Y]$, and therefore in GF($q$). 
		
		We have $|Y| \ge q^{-2(m-1)}|E(N)| > q^{n-2m+1}$.  Also, $r_M(Y) \le r(M) \le n+m-1$, so $|Y| > q^{-3m}q^{r(M|Y)}$. Finally, $M|Y$ is GF($q$)-representable, so $r_M(Y) \ge n-2m+2 \ge f_{\ref{weakroundnessreduction}}(q^2,q^{1-3m},r)$ by our first lower bound on $|Y|$. The function $g(i)$ defined by $g(i) = q^{i-3m}$ satisfies the hypotheses of Lemma~\ref{weakroundnessreduction} with $\alpha = q^{1-3m}$, so by this lemma, $M|Y$ has a weakly round restriction $R$ with $r(R) \ge r$, and $\elem(R) > q^{r(R)-3m}$.
		
		$A[E(R)]$ is a submatrix of $A[Y]$, so $R$ is a $GF(q)$-represented restriction of $M$. We have 
		\[\elem(R) > q^{-3m}q^{r(R)} \ge f_{\ref{densitygk}}(q - \tfrac{1}{2},q,m)(q - \tfrac{1}{2})^{r(R)}, \]
		so $R$ has a PG$(q',m-1)$-minor for some prime power $q' > q - \tfrac{1}{2}$. Since $R$ is GF$(q)$-representable, we must have $q' = q$, so $R$ satisfies the claim. 
	\end{proof}
	
	Let $M'$ be minor-minimal subject to the following conditions:
	\begin{itemize}
		\item $M'$ is a weakly round minor of $M$, and
		\item $\elem(M') > |\PG^{(k)}(r(M')-1,q)|$, and
		\item $R$ is a $\GF(q)$-represented restriction of $M'$. 
	\end{itemize}
	If $R$ is spanning in $M'$, then $M'$ and $R$ satisfy outcome~(\ref{spanningoutcome}). We may therefore assume that $r(R) < r(M')$. Since $R$ has a $\PG(m-1,q)$-minor, the following claim will give outcome~(\ref{constellationoutcome}) for $M'$.
	\begin{claim}
		$M'$ has an $(m,q,k+1)$-constellation restriction.
	\end{claim}
	\begin{proof}[Proof of claim:]
		We have $m \le r(R') \le r(M)-1$, so by weak roundness of $M'$, the set $E(M') - \cl_{M'}(E(R))$ has rank at least $r(M)-1 \ge m$ in $M$; let $S$ be an independent set of size $m$ in $M$, disjoint from $\cl_{M'}(E(R))$. 
		
		 For each $e \in S$, the matroid $M' \con e$ is weakly round, and we have $R = (M' \con e)|(E(R))$, so $R$ is a $\GF(q)$-represented restriction of $M' \con e$. By minimality of $M'$, it follows that 
		 \[\elem(M' \con e) \le |\PG^{(k)}(r(M' \con e)-1,q,k)|.\]
	The remainder of the proof is very similar to that of \ref{getconstclaim}.
	\end{proof}
	\end{proof}
	
	We can now prove Theorem~\ref{getepg}, which we restate here for convenience: 
	\begin{theorem}
	There is an integer-valued function $f_{\ref{getepg}}(n,q,k)$ satisfying the following: if $q$ is a prime power, $n$ and $k$ are integers with $0 \le k < n$, and $M$ is a GF$(q^2)$-representable matroid with $r(M) \ge f_{\ref{getepg}}(n,q,k)$ and 
	\[\elem(M) > |\PG^{(k)}(r(M)-1,q)|,\]
	then $M$ has a $\PG^{(k+1)}(n-1,q)$-minor.
	\end{theorem}
	\begin{proof}
		 We define the function $f_{\ref{getepg}}$ as follows. Let \[m = \max(f_{\ref{spanningwin}}(n,q,k), f_{\ref{constellationwin}}(n,q,k)).\] Let $\alpha = f_{\ref{densitygk}}(q-\tfrac{1}{2},q^2,m)$. Let  $r$ be an integer large enough so that 
		 \[|\PG^{(k)}(r'-1,q,k)| \ge \alpha(q - \tfrac{1}{2})^{r'}\]
		 for all $r' \ge r$, and let $s= f_{\ref{weakroundnessreduction}}(q^2,1,r)$. We set $f_{\ref{getepg}}(n,q,k) = s$. 
		 
	Let $M$ be a GF$(q^2)$-representable matroid with $r(M) \ge s$, and $\elem(M) > |\PG^{(k)}(r(M)-1,q)|$. The function $g(i) = |\PG^{(k)}(i-1,q)|$ can easily be seen to satisfy $g(1) \ge 1$ and $g(i) \ge 2g(i-1)$ for all $i \ge 2$, so By Lemma~\ref{weakroundnessreduction}, $M$ has a weakly round restriction $N$ with $r(N) \ge r$, and $\elem(N) > |\PG^{(k)}(r(N)-1,q)|$. 
	
	By Lemma~\ref{densitygk} and definition of $r$, $N$ has a PG$(m-1,q')$-minor for some $q' > q - \tfrac{1}{2}$. Since $N$ is GF($q^2$)-representable, we have $q' = q^2$ or $q' = q$, so in either case, $N$ has a $\PG(m-1,q)$-minor. The lemma now follows by applying Lemma~\ref{mainreduction} to $N$, and then either Lemma~\ref{spanningwin} or Lemma~\ref{constellationwin} to the minor $M'$ of $N$ given by Lemma~\ref{mainreduction}. 
	\end{proof}

\section{The Main Theorems}

	We first prove Theorem~\ref{mainresult}, which we restate here:

	\begin{theorem}
	Let $q$ be a prime power. If $\cM$ is a proper minor-closed subclass of the $\GF(q^2)$-representable matroids containing all simple $\GF(q)$-representable matroids, then there is an integer $k \ge 0$ such that $\cP_{q,k} \subseteq \cM$, and $h_{\cM}(n) = h_{\cP_{q,k}}(n)$ for all large $n$.
	\end{theorem}
	\begin{proof}
		Since $\cM$ does not contain all $\GF(q^2)$-representable matroids, there is an integer $s$ so that $\PG(s,q^2) \notin \cM$.
The set $\cP_{q,0}$ is just the set of projective geometries over $\GF(q)$, so $\cP_{q,0} \subseteq \cM$. By Corollary~\ref{squarefieldpg} and Lemma~\ref{extendedproj}, we have $\cP_{q,s'} \not\subseteq \cM$ for all $s' \ge s$; let $k \ge 0$ be maximal so that $\cP_{q,k} \subseteq \cM$. 
		 
		 We have $h_{\cM}(n) \ge h_{\cP_{q,k}}(n)$ for all $n$; we need to show that this holds with equality for all large $n$. Suppose that this is not the case. For all integers $m > k$, there is therefore some $M \in \cM$ such that $r(M) \ge f_{\ref{getepg}}(m,q,k)$ and $\elem(M) > h_{\cP_{q,k}}(r(M)) = |\PG^{(k)}(r(M)-1,q)|$. By Theorem~\ref{getepg}, $M$ therefore has an $\PG^{(k+1)}(m-1,q)$-minor. Thus, $\cM$ contains $\PG^{(k+1)}(m-1,q)$ for all $m > k$, so by Lemma~\ref{extendedproj}, $\cP_{q,k+1} \subseteq \cM$, contradicting maximality of $k$.
	\end{proof}

	Theorem~\ref{mainresult1} is now immediate, and Theorem~\ref{mainresult2} follows by applying Corollary~\ref{squarefieldpg} and Lemma~\ref{extendedproj}. Theorems~\ref{maincor1} and~\ref{maincor2} also have easy proofs:
		
		
	\begin{proof}[Proof of Theorem~\ref{maincor1}]
		Let $n_q$ be the integer $n_{1,q}$ given by Theorem~\ref{mainresult2}. By Lemma~\ref{anyfield}, $\cM$ contains $\PG^{(1)}(n-1,q)$ for all $n \ge 0$, but not $\PG(2,q^2)$; Theorem~\ref{mainresult2} gives \[ \elem(M) \le \frac{q^{r(M)+1}-1}{q-1} - q = |\PG^{(1)}(r(M)-1,q)|\] for all $M$ satisfying $r(M) \ge n_q$. But $h_{\cM}(n) \ge |\PG^{(1)}(n-1,q)|$ for all $n$, so the theorem follows.
	\end{proof}
	\begin{proof}[Proof of Theorem~\ref{maincor2}]
	Let $n_{1,q}$ be given by Theorem~\ref{mainresult2}. Let $\fH_q$ be the set of integer-valued functions $f$ so that $0 \le f(n) \le \frac{q^{2n}-1}{q^2-1}$ for all $0 \le n < n_{1,q}$, and 
	\[f(n) = \frac{q^{n+1}-1}{q-1} - q\]
	for all $n \ge n_{1,q}$. The set $\fH_q$ is clearly finite. Let $\cF$ be a set of fields satisfying the hypotheses, and $\cM$ be the class of matroids representable over all fields in $\cF$. There is some $\bF \in \cF$ with no $\GF(q^2)$-subfield, so by Lemma~\ref{anyfield}, we know that $\PG^{(1)}(n-1,q) \in \cM$ for all $n$, and $\PG(2,q^2) \notin \cM$. It now follows from $\GF(q^2)$-representability of matroids in $\cM$, and a similar argument to the proof of Theorem~\ref{maincor1}, that $h_{\cM} \in \fH_q$, giving the theorem. 
	\end{proof}

\section*{Acknowledgements}

		I would like to thank my supervisor Jim Geelen for suggesting the problem, for his useful advice towards its solution, and for his comments on the manuscript. 	
\section*{References}
\newcounter{refs}

\begin{list}{[\arabic{refs}]}
{\usecounter{refs}\setlength{\leftmargin}{10mm}\setlength{\itemsep}{0mm}}

\item\label{gk}
J. Geelen, K. Kabell,
Projective geometries in dense matroids, 
J. Combin. Theory Ser. B 99 (2009), 1-8.

\item\label{gn}
J. Geelen, P. Nelson, 
The number of points in a matroid with no n-point line as a minor, 
J. Combin. Theory. Ser. B 100 (2010), 625-630.

\item\label{bl}
T.H. Brylawski, T.D. Lucas, 
Uniquely representable combinatorial geometries,
Atti dei Convegni Lincei 17, Tomo I (1976), 83-104.

\item\label{lovasz}
L. Lov\'asz,
Selecting independent lines from a family of lines in a space,
Acta Sci. Math. 42 (1980), 121-131.

\item\label{gkw}
J. Geelen,ÊJ.P.S. Kung,ÊG. Whittle,Ê
Growth rates of minor-closed classes of matroids,
J. Combin. Theory. Ser. B 99 (2009), 420~427ÊÊÊ


\item\label{kung}
J.P.S. Kung,
Extremal matroid theory, in: Graph Structure Theory (Seattle WA, 1991), 
Contemporary Mathematics, 147, American Mathematical Society, Providence RI, 1993, pp.~21--61.

\item\label{fields}
R. Lidl, H. Niederreiter, 
Finite Fields, 
Cambridge University Press, New York, 1997.

\item \label{oxley}
J. G. Oxley, 
Matroid Theory,
Oxford University Press, New York, 2011.
\end{list}
\end{document}